\documentclass[10pt,bezier]{article}
\usepackage{amsmath,amssymb,amsfonts,graphicx,caption,subcaption}
\usepackage[colorlinks,linkcolor=red,citecolor=blue]{hyperref}

\newtheorem{prethm}{{\bf Theorem}}[section]
\newenvironment{thm}{\begin{prethm}{\hspace{-0.5
em}{\bf.}}}{\end{prethm}}
\newtheorem{prepro}{{\bf Theorem}}

\newtheorem{precor}[prethm]{{\bf Corollary}}
\newenvironment{cor}{\begin{precor}{\hspace{-0.5
em}{\bf.}}}{\end{precor}}
\newtheorem{preconj}[prethm]{{\bf Conjecture}}
\newenvironment{conj}{\begin{preconj}{\hspace{-0.5
em}{\bf.}}}{\end{preconj}}
\newtheorem{preremark}[prethm]{{\bf Remark}}
\newenvironment{remark}{\begin{preremark}\em{\hspace{-0.5
em}{\bf.}}}{\end{preremark}}
\newtheorem{prelem}[prethm]{{\bf Lemma}}
\newenvironment{lem}{\begin{prelem}{\hspace{-0.5
em}{\bf.}}}{\end{prelem}}
\newtheorem{preque}[prethm]{{\bf Question}}

\newtheorem{preobserv}[prethm]{{\bf Observation}}

\newtheorem{preproposition}[prethm]{{\bf Proposition}}

\newtheorem{preproof}{{\bf Proof.}}
\newtheorem{preprooff}{{\bf Proof}}

\newenvironment{proof}[1]{\begin{preproof}{\rm
#1}\hfill{$\Box$}}{\end{preproof}}

\newtheorem{preproofs}{{\bf Second proof of }}

\newtheorem{preprooft}{{\bf Third proof of }}

\newtheorem{preproofF}{{\bf Proof of}}

\title{\bf\Large 
Spanning trees and spanning closed walks with small degrees
}
\author{{\normalsize{\sc Morteza Hasanvand${}$} }\vspace{3mm}
\\{\footnotesize{${}$\it Department of Mathematical
 Sciences, Sharif
University of Technology, Tehran, Iran}}
{\footnotesize{}}\\{\footnotesize{  $\mathsf{hasanvand@alum.sharif.edu }$ }}}

\date{}
\begin{document}
\maketitle
\begin{abstract}{
Let $G$ be a graph and let $f$ be a positive integer-valued function on $V(G)$. In this paper, we show that if for all $S\subseteq V(G)$, $\omega(G\setminus S)<\sum_{v\in S}(f(v)-2)+2+\omega(G[S])$, then $G$ has a spanning tree $T$ containing an arbitrary given matching such that for each vertex $v$, $d_T(v)\le f(v)$, where $\omega(G\setminus S)$ denotes the number of components of $G\setminus S$ and $\omega(G[S])$ denotes the number of components of the induced subgraph $G[S]$ with the vertex set $S$. This is an improvement of several results. Next, we prove that if for all $S\subseteq V(G)$, $\omega(G\setminus S)\le \sum_{v\in S} (f(v)-1)+1$, then $G$ admits a spanning closed walk passing through the edges of an arbitrary given matching meeting each vertex $v$ at most $f(v)$ times. This result solves a long-standing conjecture due to Jackson and Wormald (1990). 
} {\small
}
\end{abstract}
%
%
%
%
%
%
%
%
%
%
%
%
%
%
\section{Introduction}
In this article, all graphs have no loop, but multiple edges are allowed and a simple graph is a graph without multiple edges.
 Let $G$ be a graph. 
The vertex set, the edge set, the maximum degree, and the number of components of $G$ are denoted by $V(G)$, $E(G)$, $\Delta(G)$, and $\omega(G)$, respectively. 
The degree $d_G(v)$ of a vertex $v$ is the number of edges of $G$ incident to $v$.
Let $F$ be a subgraph of $G$.
For  an edge set  $E$, we denote by $F-E$ the graph obtained from $F$ by removing  the edges of $E$ from $F$.
Likewise, we denote by $F+E$ the graph obtained from $F$ by inserting  the edges of $E$ into $F$.
For convenience, we use $e$ instead of $E$ when $E=\{e\}$.
For two edge sets $E_1$ and $E_2$, we also use the notation $E_1+E_2$ for the union of them.
For a vertex $v$, we denote by $d_G(v,F)$ 
the number of edges $uv$ of $G$ such that $u$ and $v$ are not in the same component of $F$.
The graph $F$ is said to be {\bf trivial}, if it has no edge. 
The graph obtained from $G$ by contracting any component of $F$ is denoted by $G/F$. 
Let $S\subseteq V(G)$. 
We denote by $G[S]$ the induced subgraph of $G$ with the vertex set $S$ containing
precisely those edges 
of $G$ whose ends lie in~$S$.
 The graph obtained from $G$ by removing all vertices of $S$ is denoted by $G\setminus S$.
We denote by $G\setminus [S,F]$ the graph obtained from $G$ by removing all edges incident to the vertices of $S$ except the edges of $F$. 
Note that while the vertices of $S$ are deleted in $G\setminus S$, no vertices are removed in $G\setminus [S, F]$.
The set $S$ is called  {\bf independent}, if there is no edge of $G$ connecting vertices in $S$.
We denote by $e_G(S)$ the number of edges of $G$ with both ends in $S$. 
Moreover, the number of edges of $G$ with both ends in $S$ joining different components of $F$ is denoted by $e_G(S,F)$.
Let  $g$ and $f$ be two nonnegative integer-valued functions on $V(G)$.
A {\bf $(g,f)$-factor} of $G$ is a spanning subgraph $H$ such that for each vertex $v$, $g(v)\le d_H(v)\le f(v)$.
For a set $A$ of integers, an {\bf $A$-factor} is a spanning subgraph with vertex degrees in $A$. 
A tree (forest) $T$ is said to be an {\bf $f$-tree ($f$-forest)},
if for each vertex $v$, $d_T(v)\le f(v)$. 
Likewise, an {\bf $f$-walk ($f$-trail)} in a graph refers to a walk (trail)
meeting each vertex $v$ at most $f(v)$ times.
A graph is called {\bf $K_{1,n}$-free}, if it has no induced subgraph isomorphic to the complete bipartite graph $K_{1,n}$. 
For a positive real number $t$, 
a graph $G$ is said to be {\bf $t$-tough}, if $\omega(G\setminus S)\le \max\{1,\frac{1}{t}|S|\}$ for all $S\subseteq V(G)$. 
A graph is called {\bf$k$-tree-connected}, if it has $k$ edge-disjoint spanning trees. 
Throughout this article, all variables 
 $k$ are positive integers.
%
%
%
%
%
%
%

In $1976$ Frank and Gy{\'a}rf{\'a}s investigated orientations of graphs with bounded out-degrees on certain connectivity properties. A special case of their result can conclude the following theorem.
\begin{thm}{\rm (\cite{MR519276, MR2746831})}\label{thm:Frank,Gyarfas}
Let $G$ be a graph with an independent set $X\subseteq V(G)$.
 If for all $S\subseteq X$,
 $\omega(G\setminus S)\le \sum_{v\in S}(f(v)-1)+1$,
then $G$ has a spanning tree $T$ such that for each $v\in X$, $d_T(v)\le f(v)$,
where  $f$ is a positive integer-valued function on $X$.
\end{thm}

In $1989$ Win~\cite{MR998275} established a result related to spanning trees and toughness of graphs, 
 and Ellingham, Nam, and Voss~($2002$) generalized it as the following. Former, Ellingham and Zha~($2000$)~\cite{MR1740929} found the following fact for constant function form. A a consequence, every $1$-tough graph must have a spanning $3$-tree containing an arbitrary given perfect matching.
\begin{thm}{\rm (\cite{MR1871346})}\label{thm:Ellingham,Nam,Voss}
{Let $G$ be a connected 
graph with a spanning forest $F$ of which every component contains at least $c$ vertices.
 Let $f$ be a positive integer-valued function on $V (G)$.
If for all $S\subseteq V(G)$, $\omega(G\setminus S)\le \sum_{v\in S}(f(v)-2)+2$, then $G$ has a spanning tree $T$ containing $F$ such that for each vertex $v$, 
$$d_T(v)\le 
 \begin{cases}
f(v)+d_F(v),	&\text{if $c = 1$};\\ 
f(v)+d_F(v)-1,	&\text{if $c \ge 2$}.
\end {cases}$$
}\end{thm}

Liu and Xu (1998) and Ellingham, Nam, and Voss (2002) independently investigated spanning trees with small degrees in highly edge-connected graphs and found the following theorem. 
\begin{thm}{\rm (\cite{MR1871346, MR1621287})}\label{thm: Liu,Xu; Ellingham, Nam, Voss}
{Every $k$-edge-connected 
simple 
graph $G$ has a spanning tree $T$ such that for each vertex $v$, $d_T(v)\le \lceil \frac{d_G(v)}{k}\rceil+2 $.
}\end{thm}

Recently, the present author (2015) refined Theorem~\ref{thm: Liu,Xu; Ellingham, Nam, Voss} and concluded the next theorems.
\begin{thm}{\rm (\cite{MR3336100})} \label{thm:k-edge}
{Every $k$-edge-connected graph $G$ has a spanning tree $T$ such that for each vertex~$v$, $d_T(v)\le \lceil \frac{d_G(v)-2}{k}\rceil+2 $.
}\end{thm}
\begin{thm}{\rm (\cite{MR3336100})}\label{thm:k-tree}
{Every $k$-tree-connected graph $G$ has a spanning tree $T$ such that for each vertex $v$, $d_T(v)\le \lceil \frac{d_G(v)-1}{k}\rceil+1$.
}\end{thm}
%
%
%

In this paper, we improve Theorems~\ref{thm:Ellingham,Nam,Voss} to the following  stronger version
which can conclude Theorems~\ref{thm:Frank,Gyarfas},~\ref{thm:Ellingham,Nam,Voss},~\ref{thm:k-edge}, and~\ref{thm:k-tree} (not necessarily directly). In particular, it can conclude that every $1$-tough graph must have a spanning $3$-tree containing an arbitrary given (not necessarily perfect) matching.
\begin{thm}\label{intro:thm:improvement}
{Let $G$ be a graph with a spanning forest $F$. Let $f$ be a positive integer-valued function on $V (G)$. If for all $S\subseteq V(G)$, 
$\omega(G\setminus S)< \sum_{v\in S}\big(f(v)-2\big)+2+\omega(G[S])$,
then  $G$ has a spanning tree $T$ containing $F$ such that for each vertex $v$,
$$d_T(v)\le 
 \begin{cases}
f(v)+d_F(v),	&\text{if $d_F(v) = 0$};\\ 
f(v)+d_F(v)-1,	&\text{if $d_F(v)\ge 1$}.
\end {cases}$$
}\end{thm}

Ellingham and Zha (2000) established a sufficient toughness condition for extending a spanning forest with non-trivial components
to a spanning tree only by inserting a matching. 
Later, Ellingham, Nam, and Voss (2002) developed their result to the following theorem.
\begin{thm}{\rm (\cite{MR1871346})}\label{thm:tough-enough:k=c}
{Let $G$ be a connected graph with a spanning forest $F$ 
of which every component contains at least $c$ vertices with $c \ge 2$.
Let $h$ be a nonnegative integer-valued function on $V(G)$.
If for every $S\subseteq V(G)$, at least one of the following conditions holds:
\begin{enumerate}{
\item [$\bullet$] 
$\omega (G\setminus S) < \sum_{v\in S}\big(\frac{1}{2}h(v)-\frac{1}{c}\big)+2$.
\item [$\bullet$] 
$\omega (G\setminus S) < \sum_{v\in S}\frac{\rho-2}{2\rho-2}h(v)+2+\frac{1}{\rho-1}$  and $ \min\{h(v):v\in S\}> 0$, where $\rho=c\min\{h(v):v\in S\}$.
}\end{enumerate}
then $G$ has a spanning tree $T$ containing $F$  such that for each vertex $v$, $d_T(v) \le h(v)+d_F(v)$.
}\end{thm}

In this paper, we provide  a common improvement for both items of Theorem~\ref{thm:tough-enough:k=c}
 as the following stronger version. 
More generally, 
we will introduce a combined version 
for this result and 
Theorem~\ref{intro:thm:improvement}.
Owing to its complicated form, we postpone it until Section~\ref{sec:toughness}.
\begin{thm}\label{Intro:thm:impliesimproves}
{Let $G$ be a graph with a spanning forest $F$ of which every component contains at least $c$ vertices with $c \ge 2$.
Let $h$ be a nonnegative integer-valued function on $V(G)$.
If for all $S\subseteq V(G)$,
$$ \omega (G\setminus S)
 <
 \sum_{v\in S}\big(\frac{c}{2c-2}h(v)-\frac{1}{c-1}\big)+2+\frac{1}{c-1}\omega(G[S]),$$
then $G$ has a spanning tree $T$ containing $F$  such that for each vertex $v$, $d_T(v) \le h(v)+d_F(v)$.
}\end{thm}

In 1990 Jackson and Wormald~\cite{MR1126990}
conjectured that every $\frac{1}{n-1}$-tough 
graph with $n\ge 2$ admits a spanning closed $n$-walk.
They also observed that this conjecture is true for $\frac{1}{n-2}$-tough graphs, when $n\ge 3$.
 Later, Ellingham and Zha (2000) proved the remaining case $n=2$
 for $4$-tough graphs by making the following result.
In Section~\ref{sec:walks}, we solve this conjecture completely
 as mentioned in the abstract.
\begin{thm}\label{intro:thm:walk:tough}{\rm (\cite{MR1740929})}
{Every $4$-tough simple graph of order at least three admits a connected $\{2,3\}$-factor containing an arbitrary given $2$-factor.
}\end{thm}
%
%
%
%
%
%
%
%
%
%
%
%
%
%
%
%
%
%
%
%
\section{Preliminary result}
Here, we state the following fundamental theorem which was studied in~\cite[Theorem 1]{MR1871346}
 for the case that $M$ is the trivial matching. 
\begin{thm}\label{thm:Preliminary}
{Let $G$ be a graph with a spanning forest $F$, and 
let $M$ be a matching of $G$ whose non-trivial components are vertex-disjoint from non-trivial components of $F$.
Let $h$ be a nonnegative integer-valued function on $V(G)$.
 If $\mathcal{T}$ is a spanning $(h+d_F)$-forest of $G$ containing $F\cup M$ with the minimum $\omega(\mathcal{T})$, then there exists a subset $S$ of $V(G)$ with the following properties:
\begin{enumerate}{
\item $\omega(G\setminus [S,F])=\omega(\mathcal{T}\setminus [S,F])$.\label{condition:final:1}
\item For each vertex $v$ of $S$, $d_\mathcal{T}(v)= h(v)+d_F(v)$.\label{condition:final:2}
}\end{enumerate}
}\end{thm}
\begin{proof}
{Define $V_0=\emptyset $.
For any $S\subseteq V(G)$ and $u\in V(G)\setminus S$, 
let $\mathcal{A}(S,u)$ be the set of all spanning $(h+d_F)$-forests 
$\mathcal{T}^\prime $ of $G$ containing $F\cup M$ such that  
$\omega(\mathcal{T}')=\omega(\mathcal{T})$, and also
$\mathcal{T}^\prime$ and $\mathcal{T}$ have the same edges, except for some of the edges of $G$ whose ends are in $V(C)\setminus S$, where $C$ is the component of $\mathcal{T}\setminus [S,F]$ containing $u$.
Now, for each positive integer $n$, recursively define $V_n$ as follows:
$$V_n=V_{n-1}\cup\{\, v\in V(G)\setminus V_{n-1}\colon \, d_{\mathcal{T}^\prime }(v)= h(v)+d_F(v) \text{\, for all\, }\mathcal{T}^\prime \in \mathcal{A}(V_{n-1},v)\,\}.$$
 Now, we prove the following claim.
%
\vspace{2mm}\\
{\bf Claim.} Let $x$ and $y$ be two vertices in different components of $\mathcal{T}\setminus [V_{n-1},F]$. 
If $xy\in E(G)\setminus E(\mathcal{T})$, then $x\in V_{n}$ or $y\in V_{n}$.
\vspace{2mm}\\
{\bf Proof of Claim.} 
By induction on $n$.
Suppose, to the contrary,  that $x$ and $y$ are in different components of 
$\mathcal{T}\setminus [V_{n-1},F]$, $xy\in E(G)\setminus E(\mathcal{T})$, and $x,y\not \in V_{n}$. 
Let $X$ and $Y$ be the vertex sets of the components of $\mathcal{T}\setminus [V_{n-1},F]$
containing  $x$ and $y$, respectively.
Since $x,y\not\in V_{n}$,
 there exist
 $\mathcal{T}_x\in \mathcal{A}(V_{n-1},x)$ and
 $\mathcal{T}_y\in \mathcal{A}(V_{n-1},y)$ with $d_{\mathcal{T}_x}(x)< h(x)+d_F(x)$ and $d_{\mathcal{T}_y}(y)< h(y)+d_F(y)$. 
For $n=1$, define $\mathcal{T}'$ to be the spanning forest of $G$ containing $F\cup M$ with
$$E(\mathcal{T}')=E(\mathcal{T})+xy-E(\mathcal{T}[X])+E(\mathcal{T}_x[X])-E(\mathcal{T}[Y])+E(\mathcal{T}_y[Y]).$$
Since $\mathcal{T}'$ is a spanning $(h+d_F)$-forest and $\omega(\mathcal{T'})<\omega(\mathcal{T})$, we arrive at a contradiction.
Now, suppose $n\ge 2$.
By the induction hypothesis, $x$ and $y$ are in the same component of $\mathcal{T}\setminus [V_{n-2},F]$. 
 Let $P$ be the unique path connecting $x$ and $y$ in $\mathcal{T}$. 
 Notice that the vertices of $P$ lie in the same component of $\mathcal{T}\setminus [V_{n-2},F]$.
Pick $e\in E(P)\setminus E(F)$ such that $e$ is incident to a vertex $z\in V_{n-1}\setminus V_{n-2}$.
According to the assumption on $M$ and $F$,
 if $e \in E(M)$ then for the other edge $e'\in E(P)\setminus \{e\}$ incident to $z$, we must have $e'\not\in E(F)\cup E(M)$.
We may therefore assume that $e \not\in E(M)$.
Now, let $\mathcal{T}'$ be the spanning forest of $G$ containing $F\cup M$ with
 $$E(\mathcal{T}')=E(\mathcal{T})-e+xy
-E(\mathcal{T}[X])+E(\mathcal{T}_x[X])
-E(\mathcal{T}[Y])+E(\mathcal{T}_y[Y]).$$
It is not hard to check that $d_{\mathcal{T}'}(z)<d_{\mathcal{T}}(z)\le h(z)+d_F(z)$ and $\mathcal{T}'$ lies in $\mathcal{A}(V_{n-2},z)$.
Since $z\in V_{n-1}$, we arrive at a contradiction.
Hence the claim holds.
%
%

Obviously, there exists a positive integer $n$ with $V_0\subseteq \cdots\subseteq V_{n-1}=V_{n}$.
 Put $S=V_{n}$.
For each $v\in V_i\setminus V_{i-1}$, we have $\mathcal{T}\in \mathcal{A}(V_{i-1},v)$ and so $d_\mathcal{T}(v)= h(v)+d_F(v)$. 
This implies  Condition~\ref{condition:final:2}.
In addition, by the previous claim,  every edge of $E(G)\setminus E(\mathcal{T})$ joining different components of 
$\mathcal{T}\setminus [S,F]$ must be incident to $S$. This establishes Condition~\ref{condition:final:1} and completes the proof.
}\end{proof}
When $F$ is the trivial spanning forest, Theorem~\ref{thm:Preliminary} can be reformulated to the following simpler version.
For our purposes in Section~\ref{section:max(0,dF-1)}, this special case would be sufficient. 
\begin{cor}\label{cor:Preliminary}
{Let $G$ be a graph with a matching $M$ and let $f$ be a positive integer-valued function on $V(G)$. 
 If $\mathcal{T}$ is a spanning $f$-forest of $G$ containing $M$ with the minimum $\omega(\mathcal{T})$, then there exists a subset $S$ of $V(G)$ with the following properties:
\begin{enumerate}{
\item $\omega(G\setminus S)=\omega(\mathcal{T}\setminus S)$.\label{cor:condition:final:1}
\item For each vertex $v$ of $S$, $d_\mathcal{T}(v)= f(v)$.\label{cor:condition:final:2}
}\end{enumerate}
}\end{cor}
%
%
%
%
%
%
%
%
%
%
%
%
%

\section{Connected $(d_F, f+d_F-1)$-factors}
The following lemma establishes a simple but important property of forests.
\begin{lem}\label{lem:spanningforest}
{Let $\mathcal{T}$ be a forest with a spanning forest $F$. If $S\subseteq V(\mathcal{T})$ and $\mathcal{F}=\mathcal{T}\setminus E(F)$, then 
$$\sum_{v\in S}d_{\mathcal{F}}(v)=\omega(\mathcal{T}\setminus [S,F])-\omega(\mathcal{T})+e_\mathcal{F}(S).$$
}\end{lem}
\begin{proof}
{By induction on the number of edges of $\mathcal{F}$ which are incident to the vertices in $S$. If there is no edge of  $\mathcal{F}$ incident to a vertex in $S$, then the proof is clear. Now, suppose that there exists an edge $e=uu'\in E(\mathcal{F})$ with $|S\cap \{u,u'\}|\ge 1$. Hence
\begin{enumerate}{
\item $\omega(\mathcal{T})=\omega(\mathcal{T}\setminus e)-1,$
\item $\omega(\mathcal{T}\setminus [S,F])= \omega((\mathcal{T}\setminus e)\setminus [S,F]),$

\item $e_\mathcal{F}(S)=e_{\mathcal{F}\setminus e}(S)+|S\cap \{u,u'\}|-1,$
\item $\sum_{v\in S}d_{\mathcal{F}}(v)=\sum_{v\in S}d_{\mathcal{F}\setminus e}(v)\,+|S\cap \{u,u'\}|.$
}\end{enumerate}
Therefore, by the induction hypothesis on $\mathcal{T}\setminus e$ with the spanning forest $F$ the lemma holds.
}\end{proof}
%
%
%
%
%
The following theorem is essential in this section.
\begin{thm}\label{thm:Omega(G[S,F])}
Let $G$ be a connected graph with $X\subseteq V(G)$ and with a factor $F$.
Let $\lambda\in [0,1]$ be a real number and let $\eta:X\rightarrow (\lambda,\infty)$ be a real function.
If for all $S\subseteq X$, 
$$\omega(G\setminus [S,F])<\sum_{v\in S}\big(\eta(v)-2\big)+2-\lambda e_G(S,F)+1-\lambda,$$
then $G$ has a connected factor $H$ containing $F$ such that for each $v\in X$,
 $d_H(v)\le \lceil\eta(v)-\lambda\rceil+d_F(v)-1.$
\end{thm}
\begin{proof}
{For each vertex $v$, define 
$$h(v)= 
 \begin{cases}
d_G(v)+1,	&\text{if  $v\not\in X$};\\
\lceil\eta(v)-\lambda\rceil-1,	&\text{if $v\in X$}. 
\end {cases}$$
First, suppose that $F$ is a forest. Let $\mathcal{T}$ be a spanning $(h+d_F)$-forest of $G$ containing $F$ with the minimum $\omega(\mathcal{T})$. 
Define $S$ to be a subset of $V(G)$ with the properties described in Theorem~\ref{thm:Preliminary}. 
If $S$ is empty, then $\omega(\mathcal{T})=\omega(G)=1$ and the theorem clearly holds. So, suppose $S$ is nonempty.
If $v\in V(G)\setminus X$, then $d_{\mathcal{T}}(v)\le d_G(v)<h(v)+d_F(v)$. This implies that $S\subseteq X$.
Put $\mathcal{F}=\mathcal{T}\setminus E(F)$. By
Lemma~\ref{lem:spanningforest} and Theorem~\ref{thm:Preliminary},
 $$\sum_{v\in S}h(v)
= \sum_{v\in S}d_{\mathcal{F}}(v)=\omega(\mathcal{T}\setminus [S,F])-\omega(\mathcal{T})+e_\mathcal{F}(S),$$
and so
$$
\omega(\mathcal{T}) = \omega(G\setminus [S,F])-\sum_{v\in S}h(v)+(1-\lambda) e_\mathcal{F}(S)+\lambda e_\mathcal{F}(S).
$$
Since $e_\mathcal{F}(S)\le |S|-1$
and $e_\mathcal{F}(S)\le e_G(S,F)$, by the assumption, we therefore have 
$$ 
\omega(\mathcal{T}) \le\omega(G\setminus [S,F])-\sum_{v\in S}(\eta(v)-\lambda-1)+(1-\lambda)(|S|-1) +\lambda e_G(S,F)< 2.$$
Hence  $\omega(\mathcal{T}) = 1$ and the theorem holds.
Now, suppose that $F$ is not a forest.
Remove some of the edges of the components of $F$ until the resulting graph $F'$ becomes a forest such that their components have the  same vertices.
It is enough, now, to apply the theorem on $F'$ and finally add the edges of $E(F)\setminus E(F')$ 
to that explored tree. 
}\end{proof}
The following corollary provides a necessary and sufficient condition for the existence of a spanning tree with the described properties.
\begin{cor}\label{thm:A necessary and sufficient condition-(X,F) = 0}
Let $G$ be a graph with a spanning forest $F$ and let $X\subseteq V(G)$ with $e_G(X,F) = 0$.
Then $G$ has a spanning tree $T$ containing $F$ such that for each $v\in X$, $d_T(v)\le h(v)+d_F(v)$, if and only if 
for all $S\subseteq X$, $\omega(G\setminus [S,F])\le \sum_{v\in S}h(v)\,+1$,
where $h$ is a nonnegative integer-valued function on $X$.
\end{cor}
\begin{proof}
{Assume that $G$ has a spanning tree $T$ containing $F$ such that for each $v\in X$, $d_T(v)\le h(v)+d_F(v)$. 
 Put $\mathcal{F}=T\setminus E(F)$ and 
let $S\subseteq X$. 
According to the assumption on $X$, one can conclude that $e_\mathcal{F}(S)=0$. Since for each $v\in S$, $d_{\mathcal{F}}(v)\le h(v)$, and $\omega(T)=1$, with respect to Lemma~\ref{lem:spanningforest},
$\omega(G\setminus [S,F])\le \omega(T\setminus [S,F])= \sum_{v\in S}d_{\mathcal{F}}(v)\,+1\le \sum_{v\in S}h(v)\,+1$.
To prove the converse, one can apply Theorem~\ref{thm:Omega(G[S,F])} with $\lambda=1$ and $\eta(v)=h(v)+2$. Note that $G$ is connected, because  $\omega(G\setminus [\emptyset,F])\le 1$.
}\end{proof}
This corollary shows an application of Corollary~\ref{thm:A necessary and sufficient condition-(X,F) = 0}. 
It can also be deduced from  Corollary~\ref{cor:Omega(GS)}.
\begin{cor}{\rm (\cite{MR519276}, see Page 5 in~\cite{MR2746831})}
Let $G$ be a graph with an independent set $X\subseteq V(G)$.
Then $G$ has a spanning tree $T$ such that for each $v\in X$, $d_T(v)\le f(v)$, if and only if\, 
 $\omega(G\setminus S)\le \sum_{v\in S}(f(v)-1)\,+1$ for all $S\subseteq X$, where $f$ is a positive integer-valued function on $X$. 
\end{cor}
\begin{proof}
{Apply Corollary~\ref{thm:A necessary and sufficient condition-(X,F) = 0} and  the fact that $\omega(G\setminus [S,F])=\omega(G\setminus S)+|S|$ when $F$ is the trivial spanning forest.
}\end{proof}
\begin{cor}
Let $G$ be a connected graph and let $f$ be a positive integer-valued function on $V(G)$.
If for all $S\subseteq V(G)$, 
$$\omega(G\setminus [S,F])\le \sum_{v\in S}\big(f(v)-2\big)+2,$$
then $G$ has a connected factor $H$ containing $F$ such that for each vertex $v$,
 $d_H(v)\le f(v)+d_F(v)-1.$
\end{cor}
\begin{proof}
{Apply Theorem~\ref{thm:Omega(G[S,F])} with $\eta=f$ and $\lambda=0$.
}\end{proof}
%
%
%
%
\subsection{Graphs with high essential edge-connectivity}
The following lemma provides two upper bounds on $\omega(G\setminus [S,F])$ depending on 
 two parameters of connectivity of $G/F$ and $d_G(v,F)$ of the vertices $v$ in $S$.
\begin{lem}\label{lem:Omega(G[S,F])}
{Let $G$ be a graph with a factor $F$ and let $S\subseteq V(G)$. Then
$$\omega(G\setminus [S,F])\le
 \begin{cases}
 \sum_{v\in S}(\frac{d_{G}(v,F)}{k}+1)\,-\frac{2}{k}e_G(S,F),	&\text{if $G/F$ is $k$-edge-connected and 
$S\neq \emptyset$};\\ 
 \sum_{v\in S}\frac{d_{G}(v,F)}{k}\;+1-\frac{1}{k}e_G(S,F),	&\text{if $G/F$ is $k$-tree-connected}.
\end {cases}$$
}\end{lem}
\begin{proof}
{First, assume that $G/F$ is $k$-edge-connected and $S$ is nonempty. 
Thus there are at least 
 $k{\bigr(}\omega(G\setminus [S, F])-|S|{\bigr)}$ 
edges of $G$ with exactly one end in $S$ joining different components of $G\setminus [S, F]$, 
because $S$ is nonempty and there are at least $\omega(G\setminus [S,F])-|S|$ 
components of $G\setminus [S,F]$ without any vertex of $S$.
 Note that we might have $\omega(G\setminus [S,F]) < |S|$.
On the other hand, there are $ \sum_{v\in S}d_{G}(v,F)\,-2e_G(S,F)$ edges of $G$ with exactly one end in $S$  joining different components of $F$. 
Hence we have
$$k{\bigr(}\omega(G\setminus [S, F])-|S|{\bigr)} \le \sum_{v\in S}d_{G}(v,F)-2e_G(S,F).$$
Next, assume that $G/F$ is $k$-tree-connected. 
Thus there are at least $k(\omega(G\setminus [S,F])-1)$ edges of $G$ with at least one end in $S$ joining different components of $G\setminus [S, F]$.
On the other hand, there are $ \sum_{v\in S}d_{G}(v,F)\,-e_G(S,F)$ edges of $G$ with  at least one end in $S$ joining different components of $F$. 
Hence we have
$$ k(\omega(G\setminus [S,F])-1)\le\; \sum_{v\in S}d_{G}(v,F)\;-e_G(S,F).$$
These inequalities complete the proof.
}\end{proof}
%
The following theorem generalizes Theorems~\ref{thm:k-edge} and~\ref{thm:k-tree}.
\begin{thm}
{Let $G$ be a graph with a factor $F$. Then $G$ has a connected factor $H$ containing $F$ such that for each vertex $v$, 
$$d_H(v)\le
 \begin{cases}
\big \lceil \frac{d_G(v)-d_F(v)-2}{k}\big\rceil+d_F(v)+2 ,	&\text{if $G/F$ is $k$-edge-connected};\\ 
\big \lceil \frac{d_G(v)-d_F(v)-1}{k}\big\rceil+d_F(v)+1 ,	&\text{if $G/F$ is $k$-tree-connected}.
\end {cases}$$
Furthermore, for an arbitrary given vertex $u$, the upper bound can be reduced to $\lfloor (d_G(u)-d_F(u))/k\rfloor+d_F(u)$.
}\end{thm}
\begin{proof}
{We may assume that $k\ge 2$, as the assertions trivially hold when $k=1$.
Since $G$ is connected, it is obvious that $ \omega (G\setminus [\emptyset,F])=1$. 
 Let $S$ be a nonempty subset of $V(G)$. 
If $G/F$ is $k$-edge-connected, then
by Lemma~\ref{lem:Omega(G[S,F])}, we have 
$$\omega(G\setminus [S,F])\le
 \sum_{v\in S}(\frac{d_G(v,F)}{k}+1)-\frac{2}{k}e_G(S,F)\le
 \begin{cases}
\sum_{v\in S}(\eta(v)-2)+2-\frac{2}{k}e_G(S,F)+1-\frac{3}{k},	&\text{if $u\in S$};\\ 
\sum_{v\in S}(\eta(v)-2)-\frac{2}{k}e_G(S,F),	&\text{if $u\not\in S$},
\end {cases}$$
where $\eta(u)= \frac{d_G(u)-d_F(u)+3}{k}$ and $\eta(v)= \frac{d_G(v)-d_F(v)}{k}+3$ for all $v\in V(G)\setminus u$.
If $G$ is $k$-tree-connected, then 
by Lemma~\ref{lem:Omega(G[S,F])}, we also have 
$$\omega(G\setminus [S,F])\le
\sum_{v\in S}\frac{d_G(v,F)}{k}\;+1-\frac{1}{k}e_G(S,F)\le
 \begin{cases}
\sum_{v\in S}(\eta(v)-2)+2-\frac{1}{k}e_G(S,F)+1-\frac{2}{k},	&\text{if $u\in S$};\\ 
\sum_{v\in S}(\eta(v)-2)+1-\frac{1}{k}e_G(S,F),	&\text{if $u\not\in S$},
\end {cases}$$
where $\eta(u)= \frac{d_G(u)-d_F(u)+2}{k}$ and $\eta(v)= \frac{d_G(v)-d_F(v)}{k}+2$ for all $v\in V(G)\setminus \{u\}$.
Hence the assertions follow from Theorem~\ref{thm:Omega(G[S,F])} with $\lambda\in \{2/k, 1/k\}$.
}\end{proof}
%
%
%
%
%
%
%
%
%
%
%
\section{Connected $(d_F, f+\max\{0,d_F-1\})$-factors}
\label{section:max(0,dF-1)}
Our aim in this section is to prove Theorem~\ref{intro:thm:improvement} and give several applications of it on connected factors.
We begin with the following lemma that allows us to make the proof simpler. This lemma can also develop a result due to Rivera-Campo~\cite{MR1451489}, who gave a sufficient condition 
for the existence of a  spanning tree with bounded maximum degree containing an arbitrary given matching. 
%
%
%
\begin{lem}\label{lem:maximal-matching}
{Let $G$ be a graph with a factor $F$.
If a maximal matching of $F$ can be extended to a spanning tree $T$, 
then $F$ itself can be extended to a connected factor $H$ such that for each vertex $v$,
 $$d_H(v)\le d_T(v)+\max\{0, d_F(v)-1\}.$$
}\end{lem}
\begin{proof}
{Choose a maximal matching $M$ of $F$ which can be extended to a spanning tree $T$.
Define $G_0=T\cup F$.
Let $T_0$ be a spanning tree of $G_0$ containing $M$ such that
$d_{T_0}(u)\le d_T(u)$  for all $u\in A=\{v\in V(G): d_M(v)=0\}$.
According to the maximality of $M$, the vertex set $A$ must be an independent set of $F$. 
Otherwise, we can insert a new edge of $F$ into $M$ to expand it to a large matching which is a contradiction.
Note  that $T$ is a natural candidate for $T_0$.
Consider $T_0$ with the maximum $| E(F)\cap E(T_0)|$. 
Define $H=T_0\cup  F$. We claim that $H$ is the desired factor we are looking for.
Let $v\in V(H)$.
If $d_{M}(v)= 1$ or $d_F(v)=0$, then 
$$d_{H}(v)\le d_{G_0}(v) \le d_{T}(v)+d_F(v)-d_{M}(v)= d_T(v)+\max\{0, d_F(v)-1\}.$$
So, suppose that $v\in A$ and $d_F(v)> 0$.
Define $F_0$ to be the factor of $F$ with $E(F_0)=E(F)\cap  E(T_0)$.
To complete the proof, we are going  to show that $d_{F_0}(v)>0$, which can imply that
$d_{H}(v) \le d_{T_0}(v)+d_F(v)-d_{F_0}(v)\le d_T(v)+d_F(v)-1$.
Suppose, to the contrary, that $d_{F_0}(v)= 0$.
Pick $vx\in E(F)$ so that $vx\notin E(T_0)$. 
Thus there exists an edge $vy\in E(T_0)$ such that the graph $T'_0=T_0-vy+vx$ is still connected (we might have $x=y$).
Since $d_{F_0}(v)= 0$, we must have $vy\not \in E(F)$ and so $vy\not\in E(M)$.
 Thus the spanning tree $T'_0$ must contain the edges of $M$.
\begin{figure}[h]
{\centering
\includegraphics[scale=0.75]{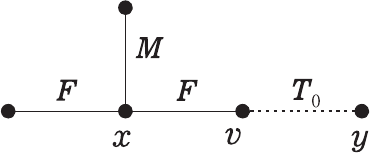}
\caption{An example for which  the vertex $x$ is the center of a star component of $F$.}
\label{Fig:12}
}\end{figure}
According to this construction,  
$d_{T'_0}(u)=d_{T_0}(u)$ for all $u\in V(G)\setminus \{x,y\}$. 
Moreover, $d_{T'_0}(x)=d_{T_0}(x)+1$ and $d_{T'_0}(y)=d_{T_0}(y)-1$ when $x\neq y$, and 
$d_{T'_0}(x)=d_{T_0}(x)$ when $x=y$.
Since $vx\in E(F)$ and $v\in A$, 
we must have $x\not \in A$. 
Therefore, $d_{T_0'}(u)\le d_{T_0}(u)\le d_{T}(u)$ for all  $u\in A$.
Since $| E(F)\cap E(T_0')|= | E(F)\cap E(T_0)|+1$, we derive a contradiction to the maximality of $T_0$, as desired.
}\end{proof}
The following theorem is essential in this section.
\begin{thm}\label{thm:c=2:G[S]}
{Let $G$ be a graph with $X\subseteq V(G)$ and with a factor $F$. 
Let $f$ be a positive integer-valued function on $X$. If for all $S\subseteq X$, 
$$\omega(G\setminus S)< \sum_{v\in S}\big(f(v)-2\big)+2+\omega(G[S]),$$ 
then  $G$ has a connected factor $H$ containing $F$ such that for each $v\in X$,
$d_H(v)\le f(v)+\max\{0, d_F(v)-1\}$.
}\end{thm}
\begin{proof}{For each $v\in V(G)\setminus X$, define $f(v)=d_G(v)+1$. Choose a matching $M$ of $F$ and 
let $\mathcal{T}$ be a spanning $f$-forest of $G$ containing $M$ with the minimum $\omega(\mathcal{T})$.
Define $S$ to be a subset of $V (G)$ with the properties described in Corollary~\ref{cor:Preliminary}.
If $v\in V(G)\setminus X$, then $d_{\mathcal{T}}(v)\le d_G(v)<f(v)$. This implies that $S\subseteq X$.
By Lemma~\ref{lem:spanningforest} and Corollary~\ref{cor:Preliminary},
$$\sum_{v\in S} f(v) = \sum_{v\in S} d_\mathcal{T}(v) = 
\omega(\mathcal{T}\setminus S)+|S| -\omega(\mathcal{T})+ e_{\mathcal{T}}(S),$$
and hence
$$
 \omega(\mathcal{T}) = \omega(G\setminus S)- \sum_{v\in S} (f(v)-1) + e_{\mathcal{T}}(S).
$$
Since 
$e_{\mathcal{T}}(S) \le |S|-\omega(G[S])$, by the assumption, we therefore have 
$$\omega(\mathcal{T}) \le \omega(G\setminus S)- \sum_{v\in S} (f(v)-2) -\omega(G[S])<2.$$
Thus $\mathcal{T}$ is a spanning $f$-tree of $G$ containing $M$. Since $M$ is an arbitrary matching, 
by Lemma~\ref{lem:maximal-matching}, one can conclude that the factor $F$ itself can be extended to a connected factor $H$ 
such that for each vertex $v$,
$d_H(v)\le f(v)+\max\{0, d_F(v)-1\}$.
 Hence the theorem is proved.
}\end{proof}
\begin{cor}\label{cor:Omega(GS)}
{Let $G$ be a graph with an independent set $X\subseteq V(G)$.
Let $f$ be a positive integer-valued function on $X$.
 If for all $S\subseteq X$, 
$$\omega(G\setminus S)\le \sum_{v\in S}(f(v)-1)+1,$$ 
then every factor $F$ can be extended to a connected factor $H$ such that for each $v\in X$,
$d_H(v)\le f(v)+\max\{0, d_F(v)-1\}.$
}\end{cor}
\begin{proof}
{Let $S$ be a subset of $X$. Since $X$ is an independent set, we must have $\omega(G[S])=|S|$ which implies that
$\omega(G\setminus S)\le \sum_{v\in S}\big(f(v)-1\big)+1< \sum_{v\in S}\big(f(v)-2\big)+2+\omega(G[S])$.
Now, it is enough to apply Theorem~\ref{thm:c=2:G[S]}.
}\end{proof}
%
%
%
%
%
%
%
%
%
Ellingham, Nam, and Voss~\cite{MR1871346} discovered the following result, when $g'$ is a positive function.
\begin{cor}\label{cor:Ellingham,Nam,Voss}
{Let $G$ be a connected graph.
If  for all $S\subseteq V(G)$, $\omega(G\setminus S)\le \sum_{v\in S}(f(v)-2)+2,$ then every $(g',f')$-factor can be extended to a connected $(g',f'+f-1)$-factor,
where $g'$ is a nonnegative
integer-valued function on $V (G)$, and $f'$ and $f$ are positive integer-valued functions on $V (G)$.
}\end{cor}
\begin{proof}
{Since $G$ is connected, it is obvious that $\omega(G\setminus \emptyset)<2+\omega(G[\emptyset])$.
Let $S$ be a nonempty subset of $V(G)$. Since $\omega(G[S])\ge 1$, we must have 
$\omega(G\setminus S)\le \sum_{v\in S}\big(f(v)-2\big)+2< \sum_{v\in S}\big(f(v)-2\big)+2+\omega(G[S])$.
Thus by Theorem~\ref{thm:c=2:G[S]}, we can extend an arbitrary given $(g',f')$-factor $F$ to a connected factor $H$ such that for each vertex $v$, $g'(v)\le d_F(v)\le d_H(v)\le f(v)+\max\{0,d_F(v)-1\}$. Since $f'(v)$ is positive, we have $d_{H}(v)\le f(v)+f'(v)-1$, regardless of $d_F(v)=0$ or not. 
Thus $H$ is the desired connected factor.
}\end{proof}
When we consider the special case $(g', f')=(0,1)$, Corollary~\ref{cor:Ellingham,Nam,Voss} becomes simpler as the following result. 
\begin{cor}\label{cor:matching}
{Let $G$ be a connected graph.
If  for all $S\subseteq V(G)$, $\omega(G\setminus S)\le \sum_{v\in S}(f(v)-2)+2$, 
then $G$ has a spanning $f$-tree containing an arbitrary given matching,
where $f$ is a positive integer-valued function on $V(G)$.
}\end{cor}
\begin{remark}
{Note that  if every matching of a graph $G$ can be extended to a spanning $f$-tree, then for all  $S\subseteq  V(G)$, we must have $\omega(G\setminus S)\le \sum_{v\in S}(f(v)-1)+1-\alpha'(G[S])$, where $\alpha'(G[S])$ denotes the
number of edges in a maximum matching of the induced graph $G[S]$. In fact, if a graph $G$ has a spanning $f$-tree $T$ containing a given forest $M$, then by Lemma~\ref{lem:spanningforest},  we clearly  have $\omega(G\setminus S)\le \omega(T\setminus S) \le \sum_{v\in S}(f(v)-1)+1-e_M(S)$  for all  $S\subseteq  V(G)$.
}\end{remark}
%
The following lemma allows us to establish the next result, and also will be employed in the last section. 
\begin{lem}{\rm (Ellingham, Nam, and Voss \cite{MR1871346})}\label{subsection:conclusions:lem}
{If $G$ is a connected $K_{1,n}$-free simple graph with $n\ge 2$, then $\omega(G\setminus S)\le (n-2)|S|+1$ for all $S\subseteq V(G)$.
}\end{lem}
Xu, Liu, and Tokuda~\cite{MR1658841} discovered the following result, when $g'$ is a positive function.
\begin{cor}\label{cor:Xu,Liu,Tokuda}
{If $G$ is a connected $K_{1,n}$-free simple graph with $n\ge 2$, then every $(g',f')$-factor can be extended to a connected $(g',f'+n-1)$-factor,
where $g'$ is a nonnegative
integer-valued function on $V (G)$ and $f'$ is a positive integer-valued function on $V (G)$.
}\end{cor}
\begin{proof}
{Apply Lemma~\ref{subsection:conclusions:lem} and Corollary~\ref{cor:Ellingham,Nam,Voss} with $f(v)=n$.
}\end{proof}
%
%
%
%
%
%
%
%
%
%
\subsection{Graphs with high edge-connectivity}
A special case of Lemma~\ref{lem:Omega(G[S,F])} can easily conclude the following lemma, because 
$e_G(S, F)=e_G(S)\ge |S|-\omega(G[S])$
and $\omega(G\setminus [S,F])=\omega(G\setminus S)+|S|$ when $F$ is the trivial spanning forest.

\begin{lem}\label{lem:Omega(GS)}
{Let $G$ be a graph with $S\subseteq V(G)$. Then
$$\omega(G\setminus S)\le
 \begin{cases}
 \sum_{v\in S}\frac{d_G(v)-2}{k}\,+\frac{2}{k}\omega(G[S]),	&\text{if $G$ is $k$-edge-connected and 
$S\neq \emptyset$};\\ 
\sum_{v\in S}(\frac{d_G(v)-1}{k}-1)+1+\frac{1}{k}\omega(G[S]),	&\text{if $G$ is $k$-tree-connected}.
\end {cases}$$
}\end{lem}
%
Another generalization of Theorems~\ref{thm:k-edge} and~\ref{thm:k-tree} is given in the next theorem.
\begin{thm}\label{thm:GS:main}
{Let $G$ be a graph with $X\subseteq V(G)$. Then every factor $F$ can be extended to a connected factor $H$ such that for each $v\in X$, 
$$d_H(v)\le \max\{0,d_F(v)-1\}+
 \begin{cases}
\lceil \frac{d_G(v)-2}{k}\rceil+2,	&\text{if $G$ is $k$-edge-connected};\\ 
\lceil \frac{d_G(v)-1}{k}\rceil+1,	&\text{if $G$ is $k$-tree-connected};\\
\lceil \frac{d_G(v)}{k}\rceil+1,	&\text{if $G$ is $k$-edge-connected and $X$ is independent};\\ 
\lceil \frac{d_G(v)}{k}\rceil,	&\text{if $G$ is $k$-tree-connected and $X$ is independent}.
\end {cases}$$
Furthermore, for an arbitrary given vertex $u$, the upper bound can be reduced to
$\lfloor d_G(u)/k\rfloor+\max\{0,d_F(u)-1\}$.
}\end{thm}
\begin{proof}
{We may assume that $k\ge 2$, as the assertions trivially hold when $k=1$.
Since $G$ is connected, it is obvious that $\omega(G\setminus \emptyset)<2+\omega(G[\emptyset])$.
Let $S$ be a nonempty subset of $X$
 so that $\omega(G[S])\ge 1$.
If $G$ is $k$-edge-connected, then
by Lemma~\ref{lem:Omega(GS)}, we have 
$$\omega(G\setminus S)\le \sum_{v\in S}\frac{d_G(v)-2}{k}+\frac{2}{k}\omega(G[S])< 
 \sum_{v\in S}(f(v)-2)+2+\omega(G[S]),$$
where $f(u)= \lceil\frac{d_G(u)+1}{k}\rceil-1$ and $f(v)= \lceil\frac{d_G(v)-2}{k}\rceil+2$ for all $v\in V(G)\setminus \{u\}$.
Note that $f(u)=\lfloor \frac{d_G(u)}{k}\rfloor \ge 1$ because of $d_G(u)\ge k$.
If $G$ is $k$-tree-connected, then 
by Lemma~\ref{lem:Omega(GS)}, we also have 
$$\omega(G\setminus S)\le \sum_{v\in S}(\frac{d_G(v)-1}{k}-1)+1+\frac{1}{k}\omega(G[S])< 
\sum_{v\in S}(f(v)-2)+2+\omega(G[S]),$$
where $f(u)= \lceil\frac{d_G(u)+1}{k}\rceil-1$ and $f(v)= \lceil\frac{d_G(v)-1}{k}\rceil+1$ for all $v\in V(G)\setminus \{u\}$.
Thus the first two assertions follow from Theorem~\ref{thm:c=2:G[S]}.
Now, suppose that $X$ is an independent set.
If $G$ is $k$-edge-connected, then
by Lemma~\ref{lem:Omega(GS)}, we have 
$$\omega(G\setminus S)\le \sum_{v\in S}\frac{d_G(v)}{k}< 
\sum_{v\in S}(f(v)-1)+2,$$
where $f(u)= \lceil\frac{d_G(u)+1}{k}\rceil-1$ and $f(v)= \lceil\frac{d_G(v)}{k}\rceil+1$ for all $v\in X\setminus \{u\}$.
If $G$ is $k$-tree-connected, then 
by Lemma~\ref{lem:Omega(GS)}, we also have 
$$\omega(G\setminus S)\le \sum_{v\in S}(\frac{d_G(v)}{k}-1)+1< 
\sum_{v\in S}(f(v)-1)+2,$$
where $f(u)= \lceil\frac{d_G(u)+1}{k}\rceil- 1$ and $f(v)= \lceil\frac{d_G(v)}{k}\rceil$ for all $v\in X\setminus \{u\}$.
Thus the second two assertions follow from Corollary~\ref{cor:Omega(GS)}.
}\end{proof}
\begin{cor}
{If $G$ is a $k$-edge-connected graph with $\Delta(G)\le k(n-2)+2$ and $n\ge 2$, then every $(g',f')$-factor can be extended to a connected $(g',f'+n-1)$-factor, where $g'$ is a nonnegative integer-valued function on $V(G)$ and $f'$ is a positive integer-valued function on $V(G)$.
}\end{cor}
\begin{proof}
{Let $F$ be a $(g',f')$-factor of $G$.
By Theorem~\ref{thm:GS:main}, we can extend $F$ to a connected factor $H$ such that for each vertex $v$, $g'(v)\le d_F(v)\le d_H(v)\le n+\max\{0,d_F(v)-1\}$. Since $f'(v)$ is positive, we have $d_{H}(v)\le n+f'(v)-1$,  regardless of  $d_F(v)=0$ or not. 
Hence $H$ is the desired connected factor.
}\end{proof}
The following corollary makes a strengthened version for Lemma 2.2 (ii) in~\cite{MR1126990}.
\begin{cor}
{Let $G$ be a graph with a matching $M$ and let $f$ be a  positive integer-valued function on $V(G)$. 
If $G$ admits a spanning closed $f$-walk passing through the edges of $M$, then
 it has
a spanning $(f+1)$-tree $T$ containing  the edges of $M$. Furthermore, for an arbitrary given vertex $u$, we can have $d_T(u)\le f(u)$.
}\end{cor}
\begin{proof}
{Let $H$ be the Eulerian graph with $V(H)=V(G)$ obtained from
a spanning closed $f$-walk of $G$ passing through the edges of $M$, 
by inserting  $t$ copies of every edge $e$ of $G$ into $H$ in which $e$ is used $t$ times in the desired walk.
 Since $H$ is Eulerian, it is $2$-edge-connected. 
Thus by Theorem~\ref{thm:GS:main}, one can conclude that the graph $H$ has a spanning $(f+1)$-tree $T$ containing the edges of $M$ with $d_T(u)\le f(u)$,
 and so does $G$.
}\end{proof}
Another strengthened version of Lemma 2.2 (ii) in~\cite{MR1126990} is given in the following theorem.
This result allows one to deduce Corollary~\ref{cor:matching} form Theorem~\ref{thm:walk:spanning}. In Section~\ref{sec:walks},
we will conversely show that Corollary~\ref{cor:matching}  implies Theorem~\ref{thm:walk:spanning},
and the two are therefore equivalent.
\begin{thm}
{Let $G$ be a graph with a matching $M$ and let $f$ be a  positive integer-valued function on $V(G)$. 
If $G$ admits a spanning $f$-walk (not necessarily closed) passing through the edges of $M$, then
 it has
a spanning $(f+1)$-tree containing  the edges of $M$. 
}\end{thm}
\begin{proof}
{Let $H$ be the connected graph with $V(H)=V(G)$ obtained from
a spanning $f$-walk of $G$ passing through the edges of $M$, by inserting  $t$ copies of every edge $e$ of $G$ into $H$ in which $e$ is used $t$ times in the desired walk. Note that  all vertices of $H$, except possibly two vertices, have even degrees. This shows that
the graph $H$ can be made $2$-edge-connected by adding at most one edge.
Let $S\subseteq V(G)$. Obviously, for all components $C$ of $H\setminus S$, there is at least one edge with exactly one edge in $V(C)$. Moreover, all components of $C$ of $H\setminus S$, except possibly two components, have at least two edges of $H$ with exactly one end in $V(C)$. Therefore, there are at least $2\omega(H\setminus S)-2$ edges of $H$ with exactly on edge in $S$.
This implies that $2\omega(H\setminus S)-2\le \sum_{v\in S}d_H(v)-2e_H(S)$ and hence 
$\omega(H\setminus S)\le \sum_{v\in S}(d_H(v)/2-1)+1+\omega(H[S])$. Thus by Theorem~\ref{thm:c=2:G[S]}, one can conclude that the graph $H$ has a spanning $(f+1)$-tree $T$ containing the edges of $M$, and so does $G$.
}\end{proof}
%
%
%
%
%
%
%
%
%
%
%
%
\section{Toughness and the existence of connected $\{r,r+1\}$-factors}
\label{sec:toughness}
\label{sec:tougn-enough}
In this section, we shall introduce a combined stronger version  for Theorems~\ref{intro:thm:improvement} and~\ref{Intro:thm:impliesimproves}.
For this purpose, we need to establish the following  lemma that  provides a  relationship between $\omega(G\setminus S)$ and $\omega(G\setminus [S,F])$. 
\begin{lem}\label{lem:Omega(G[S,F]), Omgea(GS)}
{Let $G$ be a graph with a spanning forest $F$. 
Let $c\in [2,\infty)$ be a real number and let $\xi:V(G)\rightarrow [0,1]$ be a real function 
in which for every component $C$ of $F$, 
$\sum_{v\in V(C)}\xi(v)\ge 1-\frac{1}{c-1}(|V(C)|-1)$.
If $S\subseteq V(G)$,
then $$\omega(G\setminus [S,F])\le \omega(G\setminus S)+\frac{1}{c-1}e_F(S)+
\sum_{v\in S}\xi(v).$$
}\end{lem}
\begin{proof}
{Since every component $C$ of $F$ whose vertices entirely lie in the set $S$ has exactly
$|V(C)| -1$ edges with both ends in $S$, we must have $1 \le \sum_{v\in V(C)}\xi(v)+\frac{1}{c-1}e_{C}(S)$.
Thus $t_s \le \sum_{v\in S'}\xi(v)+\frac{1}{c-1}e_{F}(S')\le \sum_{v\in S}\xi(v)+\frac{1}{c-1}e_{F}(S),$
where  $S'$ is the set of all vertices belonging to the components of $F$
 whose vertices entirely lie in the set $S$, and 
$t_s$ is the number of those components.
Therefore, $\omega(G\setminus [S,F])\le \omega(G\setminus S)+ t_s\le \sum_{v\in S}\xi(v)+\frac{1}{c-1}e_{F}(S)$.
Hence the lemma holds.
}\end{proof}
The following theorem is essential in this section.
\begin{thm}\label{thm:tough-enough}
{Let $G$ be a graph with a factor $F$ and 
let $h$ be a nonnegative integer-valued function on $V(G)$.
Let $c\in [2,\infty) $ be a real number and let $\xi:V(G)\rightarrow [0,1]$ be a real function 
in which for every component $C$ of $F$, 
$\sum_{v\in V(C)}\xi(v)\ge 1-\frac{1}{c-1}(|V(C)|-1)$.
If for every $S\subseteq V(G)$,
$$ \omega (G\setminus S) <
\sum_{v\in S}\big(h(v)-\frac{1}{c-1}-\xi(v)\big)+2+\frac{1}{c-1}\omega(G[S])-
\frac{c-2}{c-1}\min\{\lfloor\sum_{v\in S}\frac{h(v)}{2}\rfloor,|S|-\omega(G[S])\},$$
then $G$ has a connected factor $H$ containing $F$ such that for each vertex $v$, $d_H(v) \le h(v)+d_F(v)$.
}\end{thm}
\begin{proof}{
First, suppose that $F$ is a forest.
Let $\mathcal{T}$ be a spanning $(h+d_F)$-forest of $G$ containing $F$ with the minimum $\omega(\mathcal{T})$.
Define $S$ to be a subset of $V (G)$ with the properties described in Theorem~\ref{thm:Preliminary}.
Put $\mathcal{F}=\mathcal{T}\setminus E(F)$.
By Lemma~\ref{lem:spanningforest} and Theorem~\ref{thm:Preliminary},
$$\sum_{v\in S} h(v) = \sum_{v\in S} d_\mathcal{F}(v) = 
\omega(\mathcal{T}\setminus [S, F]) -\omega(\mathcal{T})+ e_{\mathcal{F}}(S),$$
and so
\begin{equation}\label{eq:thm:2-connected:0}
 \omega(\mathcal{T}) = \omega(G\setminus [S, F])- \sum_{v\in S} h(v) + e_{\mathcal{F}}(S).
\end{equation}
Also, by Lemma~\ref{lem:Omega(G[S,F]), Omgea(GS)},
\begin{equation}\label{eq:thm:2-connected:1}
\omega(G\setminus [S,F]) \le 
\omega(G\setminus S)+\frac{1}{c-1}e_F(S)+\sum_{v\in S}\xi(v).
\end{equation}
Since 
$e_\mathcal{F}(S)+e_F(S)=e_{\mathcal{T}}(S) \le |S|-\omega(G[S])$,
$$e_\mathcal{F}(S)+\frac{1}{c-1}e_F(S)\le \frac{c-2}{c-1}e_\mathcal{F}(S)+\frac{1}{c-1}(|S|-\omega(G[S])).$$
In addition, since 
 $e_\mathcal{F}(S)
\le\lfloor\frac{1}{2}\sum_{v\in S}d_{\mathcal{F}}(v) \rfloor=\lfloor\frac{1}{2}\sum_{v\in S}h(v)\rfloor$, we must have
\begin{equation}\label{eq:thm:2-connected:2}
e_\mathcal{F}(S)+\frac{1}{c-1}e_F(S)
\le \frac{c-2}{c-1}\min\{\lfloor\frac{1}{2}\sum_{v\in S}h(v)\rfloor, |S|-\omega(G[S])\}+\frac{1}{c-1}(|S|-\omega(G[S])).
\end{equation}
Therefore, Relations~(\ref{eq:thm:2-connected:0}),~(\ref{eq:thm:2-connected:1}), and~(\ref{eq:thm:2-connected:2}) can conclude that
\begin{equation*}
\omega(\mathcal{T}) \le
\omega(G\setminus S)-
\sum_{v\in S}\big(h(v)-\frac{1}{c-1}-\xi(v)\big)-\frac{1}{c-1}\omega(G[S])+
\frac{c-2}{c-1}\min\{\lfloor\sum_{v\in S}\frac{h(v)}{2}\rfloor,|S|-\omega(G[S])\}
< 2.
\end{equation*}
Hence $\omega(\mathcal{T}) =1$ and the theorem holds.
Now, suppose that $F$ is not a forest.
Remove some of the edges of the components of $F$ until the resulting graph $F'$ becomes a forest such that their components have the  same vertices.
It is enough, now, to apply the theorem on $F'$ and finally add the edges of $E(F)\setminus E(F')$ 
to that explored tree. 
}\end{proof}
\begin{cor}\label{cor:path-cycle}
{Let $G$ be a simple graph and let $F$ be a factor of $G$ with $\Delta(F)\le 2$. Then $F$ can be extended to connected factor $H$ with $\Delta(H)\le 3$,
if for all $S\subseteq V(G)$, $\omega(G\setminus S)\le \frac{1}{4}|S|+1$.
}\end{cor}
\begin{proof}
{For each vertex $v$, define $h(v)=3-d_F(v)$.
Also, define $\xi(v)=1$ when $d_F(v)=0$, and 
define $\xi(v)=1/4$ when $d_F(v)=1$, 
and $\xi(v)=0$ otherwise. According to these definitions, $1/4\le  3h(v)/4-1/2-\xi(v)$.
Thus if $S\subseteq V(G)$, then 
$ \omega (G\setminus S) \le \frac{1}{4}|S|+1<
\sum_{v\in S}\big(\frac{c}{2c-2}h(v)-\frac{1}{c-1}-\xi(v)\big)+2$,
where $c=3$.
Let $C$ be a component of $F$.
Since $F$ has no multiple edges, if $C$ has a vertex  with degree two, then it must contain at least three vertices, which implies that $\sum_{v\in V(C)}\xi(v)\ge 0\ge 1-(|V(C)|-1)/2$.
Otherwise, we again have $\sum_{v\in V(C)}\xi(v)\ge 1-(|V(C)|-1)/2$ regardless of $C$ contains two vertices with degree one or not.
 Hence it is enough to apply Theorem~\ref{thm:tough-enough} to complete the proof. 
}\end{proof}
The following corollary improves Theorem~\ref{Intro:thm:impliesimproves} and implies Theorem~\ref{intro:thm:improvement}.
\begin{cor}\label{cor:c=3:G[S]}
{Let $G$ be a graph with a factor $F$ of which every non-trivial component  contains at least $c$ vertices with $c \ge 2$.
Let $h$ be a nonnegative integer-valued function on $V(G)$.
If for all $S\subseteq V(G)$,
$$ \omega (G\setminus S) < \sum_{v\in S}\big(\frac{c}{2c-2}h(v)-\frac{1}{c-1}\big)+2+\frac{1}{c-1}\omega(G[S])-|\{v\in S:d_F(v)=0\}|,$$
then $G$ has a connected factor $H$ containing $F$  such that for each vertex $v$, $d_H(v) \le h(v)+d_F(v)$.
}\end{cor}
\begin{proof}
{For each vertex $v$, define $\xi(v)=1$ when $d_F(v)=0$, and define $\xi(v)=0$ otherwise.
Let $C$ be a component of $F$.
If  $C$ is a non-trivial component, then by the assumption, we must have $\sum_{v\in V(C)}\xi(v)\ge 0 \ge 1-(|V(C)|-1)/(c-1)$.
Thus $\sum_{v\in V(C)}\xi(v)\ge 1-(|V(C)|-1)/2$ regardless of $C$ is a trivial component or not.
 Hence it is enough to apply Theorem~\ref{thm:tough-enough} to complete the proof.
}\end{proof}
%
%
%
%
\begin{cor}\label{cor:c=2:G[S]}
{Let $G$ be a graph with a factor $F$. Let $f$ be a positive integer-valued function on $V (G)$. If for all $S\subseteq V(G)$, 
$$\omega(G\setminus S)< \sum_{v\in S}\big(f(v)-2\big)+2+\omega(G[S]),$$ 
then  $G$ has a connected factor $H$ containing $F$ such that for each vertex $v$,
$d_H(v)\le f(v)+\max\{0, d_F(v)-1\}$.
}\end{cor}
\begin{proof}
{Apply Corollary~\ref{cor:c=3:G[S]}  with $c=2$ and $h(v)=f(v)-1+\max\{0, 1-d_F(v)\}$.
}\end{proof}
We here derive the following corollary from Corollary~\ref{cor:c=3:G[S]} by comparing their conditions.
\begin{cor}{\rm (\cite{MR1871346})}\label{cor:rho}
{Let $G$ be a connected graph with a factor $F$ 
of which every component contains at least $c$ vertices with $c \ge 2$.
Let $h$ be a nonnegative integer-valued function on $V(G)$.
If for every $S\subseteq V(G)$, at least one of the following conditions holds:
\begin{enumerate}{
\item [$\bullet$] 
$\omega (G\setminus S) < \sum_{v\in S}\big(\frac{1}{2}h(v)-\frac{1}{c}\big)+2$.
\item [$\bullet$] 
$ \omega (G\setminus S) < \sum_{v\in S}\frac{\rho-2}{2\rho-2}h(v)+2+\frac{1}{\rho-1}$  and $ \min\{h(v):v\in S\}> 0$, where $\rho=c\min\{h(v):v\in S\}$.
}\end{enumerate}
then $G$ has a connected factor $H$ containing $F$  such that for each vertex $v$, $d_H(v) \le h(v)+d_F(v)$.
}\end{cor}
\begin{proof}
{Since $G$ is connected, it is obvious that $ \omega (G\setminus \emptyset) <2+\frac{1}{c-1}\omega(G[\emptyset])$.
Let $S$ be a nonempty subset of $V(G)$ and set $S_0=\{v\in S:h(v)=0\}$.
Take $S_0'$ to be a subset of $S_0$ with $|S_0|-1 \le |S_0'|<|V(G)|$.
If $|S_0|\ge 1$, then the first condition of the theorem must hold for the vertex set $S_0'$.
Thus $1\le \omega (G\setminus S'_0) <-|S_0'|/c+2$. This implies that $|S'_0|< c$. Hence $|S_0|\le c$
 regardless of $S_0$ is empty or not. Thus
$$ \sum_{v\in S}\big(\frac{1}{2}h(v)-\frac{1}{c}\big)
 \le \sum_{v\in S}\big(\frac{c}{2c-2}h(v)-\frac{1}{c-1}\big)+\frac{|S_0|}{c(c-1)}\le  \sum_{v\in S}\big(\frac{c}{2c-2}h(v)-\frac{1}{c-1}\big)+\frac{1}{c-1}.$$
In addition, if $|S_0|= 0$ then
$$ \sum_{v\in S}\frac{\rho-2}{2\rho-2}h(v)=
 \sum_{v\in S}(\frac{c}{2c-2}-\frac{1}{2c-2}-\frac{1}{2\rho-2})h(v)\le 
 \sum_{v\in S}(\frac{c}{2c-2}h(v)-\frac{1}{c-1}).$$
More precisely,  $\frac{1}{2c-2}+\frac{1}{2\rho-2}=\frac{1}{c-1}$ when $\rho=c$, and
$\frac{1}{2c-2}h(v) \ge \frac{1}{c-1}$ when $\rho \ge 2c$.
Therefore, by the assumption, one can conclude that $\omega (G\setminus S) < \sum_{v\in S}\big(\frac{c}{2c-2}h(v)-\frac{1}{c-1}\big)+2+\frac{1}{c-1}\omega(G[S])$. 
Hence it is enough to apply Corollary~\ref{cor:c=3:G[S]} to complete the proof.
}\end{proof}
When we consider the special case $h=1$, Corollary~\ref{cor:rho} becomes simpler as the following version.
\begin{cor}{\rm(\cite{MR1740929})}\label{cor:insertingMatching}
{Let $G$ be a connected graph with a factor $F$ of which every component contains at least $c$ vertices with $c \ge 2$.
If for all $S\subseteq V(G)$,
$$ \omega (G\setminus S)
\le \frac{c-2}{2c-2}|S|+2+\frac{1}{2c-2},$$
then $G$ has a connected factor $H$ containing $F$  such that for each vertex $v$, $d_H(v) \in \{d_F(v),d_F(v)+1\}$.
}\end{cor}
Enomoto, Jackson, Katerinis, and Saito (1985)~\cite{MR785651} showed that every $r$-tough graph $G$ of order at least $r+1$ with $r|V(G)|$ even admits an $r$-factor. 
For the case that $r|V(G)|$ is odd, the same arguments can imply that the graph $G$ admits 
a factor such that whose degrees are $r$, except for a vertex with degree  $r+1$.
A combination of Corollary~\ref{cor:c=3:G[S]} and this result can conclude the following corollary.
\begin{cor}{\rm(\cite{MR1871346, MR1740929})}\label{cor:r-tough:r-regular}
{Every $r$-tough graph of order at least $r+1$ with $r\ge 3$ admits a connected $\{r,r+1\}$-factor.
}\end{cor}
\begin{proof}
{We may assume that $G$ is an $r$-tough simple graph, by deleting multiple edges from $G$ (if necessary). 
Let $F$ be an $\{r,r+1\}$-factor of $G$ such that each of whose vertices has degree $r$, except for at most one vertex $u$ with degree $r+1$~\cite{MR785651}. Note that every component of $F$ must contain at least $r+1$ vertices.
Let $S$ be a subset of $V(G)$. Since $G$ is $r$-tough, 
$$\omega(G\setminus S)\le
 \frac{1}{r}|S|+1\le 
\frac{r-1}{2r}|S|+1<
 \sum_{v\in S}\big(\frac{c}{2c-2}h(v)-\frac{1}{c-1}\big)+2+\frac{1}{c-1}\omega(G[S]),$$
where $c=r+1$, $h(u)=0$, and $h(v)=1$ for each vertex $v$ with $v\neq u$.
Thus by applying Corollary~\ref{cor:c=3:G[S]}, 
the graph $G$ has a connected factor $H$ containing $F$ 
such that  for each vertex $v$, $d_H(v) \le h(v) + d_F(v)$, and also $d_H(u)=d_F(u)$.
This implies that $H$ is a connected $\{r,r+1\}$-factor. 
}\end{proof}
\begin{remark}
{Note that Corollary~\ref{cor:r-tough:r-regular} can  be proved by Corollary~\ref{cor:insertingMatching},
 except for the case that $r|V(G)|$ is odd, and it
can be proved by Corollary~\ref{cor:rho}, except for the case that $r=3$ and $|V(G)|$ is odd.
}\end{remark}
\section{Applications to spanning closed walks}
\label{sec:walks}
Our aim in this section is to prove  
 a long-standing conjecture due to Jackson and Wormald~\cite{MR1126990} 
with a stronger version. Before doing so, we state some results on spanning parity forests.
%
%
%
%
%
%
%
%
\subsection{Spanning parity $f$-forests}
In 1985 Amahashi~\cite{MR951772} introduced a criterion for the existence of a spanning forest 
with bounded maximum degree in which all vertices have odd degree.
Later, Yuting and Kano (1988) generalized it by establishing the following theorem.
We denote below by $odd(G)$ the number of components of $G$ with odd order.
\begin{thm}{\rm (\cite{MR956194})}\label{thm:Yuting-Kano}
{Let $G$ be a graph and let $f$ be an odd positive integer-valued function on $V(G)$.
Then $G$ has a spanning $f$-forest with odd degree vertices if and only if for all $S\subseteq V(G)$,
$$odd(G\setminus S)\le \sum_{v\in S} f(v).$$
}\end{thm}
 Kano, Katona, and Szab\'o (2009) 
studied a more general version for Theorem~\ref{thm:Yuting-Kano} which gives a criterion for the existence of parity $f$-forests. 
We denote below by $odd_h(G)$ the number of components of $G$ with odd number of vertices $v$ with $h(v)$ odd.
\begin{thm}\label{thm:Devote}{\rm (\cite{MR2600471})}
{Let $G$ be a graph and let $h$ be a nonnegative integer-valued function on $V(G)$.
Then $G$ has a spanning $h$-forest $F$ such that for each vertex $v$, $d_F(v)$ and $h(v)$ have the same parity, if and only if
 for all $S\subseteq V(G)$,
$$odd_h(G\setminus S)\le \sum_{v\in S} h(v).$$
}\end{thm}
\begin{proof}
{The proof presented here is introduced in~\cite[Section 1]{MR2600471} implicitly.
Obviously,  if $G$ has a spanning $h$-forest $F$ such that for each vertex $v$, $d_F(v)$ and $h(v)$ have the same parity, then for every $C$ component of $G\setminus S$ with $\sum_{v\in V(C)}h(v)$ odd, there must be an edge of $F$ with one end in $V(C)$ and the other one in $S$, where $S\subseteq V(G)$.
Thus $odd_h(G\setminus S)  \le \sum_{v\in S}d_F(v)\le  \sum_{v\in S} h(v)$.
Now, it remains to prove  the sufficiency. 
Define $G_0$ to be the graph obtained from $G$ by adding a new vertex $u'$  
and a pendant edge $u'u$ corresponding to each $u\in U=\{v\in V(G): h(v)\text{  is even}\}$.
 Let us  define $f(u)=h(u)+1$ and  $f(u')=1$  for all $u\in U$, and  $f(v)=h(v)$ for all $v\in V(G)\setminus U$.
Note that $f$ is an odd positive integer-valued function on $V(G_0)$.
It is not hard to check that for every $S\subseteq V(G_0)$, 
$$odd(G_0\setminus S)\le odd_h(G\setminus S)+|S\cap (U\cup U')|,$$ 
where $U'=\{u': u\in U\}$.
This implies that 
$odd(G_0\setminus S)\le \sum_{v\in S\setminus U'} h(v)+|S\cap (U\cup U')| = \sum_{v\in S} f(v)$.
Thus by Theorem~\ref{thm:Yuting-Kano}, the graph $G_0$ has a spanning $f$-forest $F_0$ in which all vertices have odd degree. Obviously, this forest contains all inserted pendant edges and so by removing them we can find a forest $F$ of $G$ with the desired properties. 
}\end{proof}
We shall below derive a conclusion of Theorem~\ref{thm:Devote}, which will be used in the subsequent subsection.
This  result is proved in~\cite{arXiv:Kano-Lu} when $f$ is an odd positive integer-valued function.
\begin{cor}\label{cor:Devote}
{Let $G$ be a graph and let $f$ be a positive integer-valued function on $V(G)$.
Then  for all $S\subseteq V(G)$, $$\omega(G\setminus S)\le \sum_{v\in S} (f(v)-1)+1,$$
if and only if for every  $Q\subseteq V(G)$ with  even size, the graph $G$ has a spanning $f$-forest $F$ such that $Q=\{v\in V(G):d_F(v) \text{ is odd}\}$.
}\end{cor}
\begin{proof}
{We first prove the necessity. Let $Q$ be a subset of $V(G)$ with even size. For each vertex $v$, 
define $h(v)$ to be either $f(v)$ or $f(v)-1$ such that $\{v\in V(G):h(v) \text{ is odd}\}=Q$.
Clearly, $\sum_{v\in V(G)}h(v)$ is even.
Let $S\subseteq V(G)$.
By the assumption, $odd_{h}(G\setminus S)\le \omega(G\setminus S)\le  \sum_{v\in S} (f(v)-1)+1\le \sum_{v\in S}h(v)+1$.
It is easy to check that $odd_{h}(G\setminus S)+\sum_{v\in S}h(v)$ and $\sum_{v\in V(G)}h(v)$ have the same parity 
and so
 $odd_{h}(G\setminus S)$ and $\sum_{v\in S}h(v)$ have the same parity.
 Thus $odd_{h}(G\setminus S)\le \sum_{v\in S}h(v).$
Therefore, by Theorem~\ref{thm:Devote}, the graph $G$ has a 
spanning $h$-forest $F$
 such that for each vertex $v$, $d_F(v)$ and $h(v)$ have the same parity.
Hence the necessity is proved.

Now, we shall prove the sufficiency. Let $S\subseteq V(G)$. We may assume that $G\setminus S$
contains a component $C'$.
For each $v\in S$, define $h(v)=f(v)-1$, and
 for each $v\in V(G)\setminus S$,  define $h(v)$ to be either $f(v)$ or $f(v)+1$
such that for every other component $C$ of $G\setminus S$,  $\sum_{v\in V(C)}h(v)$ is odd, 
and  also $\sum_{v\in V(C')}h(v)$ and $\sum_{v\in V(G)\setminus V(C')}h(v)$ have the same parity.
Since  $\sum_{v\in V(G)}h(v)$ is even, by the assumption, $G$ has a spanning $f$-forest $F$ such that for each vertex $v$, $d_F(v)$ and $h(v)$ have the same parity.
Thus $F$ is also a spanning $h$-forest of $G$ and $\omega(G\setminus S)-odd_h(G\setminus S)\in \{0,1\}$. 
Therefore, by Theorem~\ref{thm:Devote}, 
$\omega(G\setminus S)\le odd_h(G\setminus S)+1\le \sum_{v\in S}h(v)+1=\sum_{v\in S}(f(v)-1)+1$.
Hence the proof is completed.
}\end{proof}
The following result improves the upper bounds in Theorems~\ref{thm:k-edge} and~\ref{thm:k-tree}
when the existence of parity forests are considered.
\begin{cor}\label{cor:paths}
{Let $G$ be a connected graph with $Q\subseteq V(G)$, where $|Q|$ is even. 
Then $G$ has a spanning forest $F$ such that $Q=\{v\in V(G):d_F(v) \text{ is odd}\}$, and for each vertex $v$, 
$$d_F(v)\le 
 \begin{cases}
 \big\lceil\frac{d_G(v)}{k}\big\rceil +1,	&\text{if $G$ is $k$-edge-connected};\\ 
 \big\lceil\frac{d_G(v)}{k}\rceil,	&\text{if $G$ is $k$-tree-connected}.
\end {cases}$$
Furthermore, for an arbitrary given vertex $u$, the upper bound can be reduced to $\lfloor d_G(u)/k\rfloor$.
}\end{cor}
\begin{proof}
{Since $G$ is connected, it is obvious that $ \omega (G\setminus \emptyset)= 1$. 
Let $S$ be a nonempty subset of $V(G)$.
If $G$ is $k$-edge-connected, then
by Lemma~\ref{lem:Omega(GS)}, we have 
$$\omega(G\setminus S)\le
 \sum_{v\in S}\frac{d_G(v)}{k}< 
 \sum_{v\in S}(f(v)-1)+2,$$
where $f(u)= \lceil\frac{d_G(u)+1}{k}\rceil-1$ and $f(v)= \lceil\frac{d_G(v)}{k}\rceil+1$ for all $v\in V(G)\setminus \{u\}$.
Note that $f(u)=\lfloor \frac{d_G(u)}{k}\rfloor$.
If $G$ is $k$-tree-connected, then by Lemma~\ref{lem:Omega(GS)}, we also have 
$$\omega(G\setminus S)\le \sum_{v\in S}(\frac{d_G(v)}{k}-1)+1< 
\sum_{v\in S}(f(v)-1)+2,$$
where $f(u)= \lceil\frac{d_G(u)+1}{k}\rceil-1$ and $f(v)= \lceil\frac{d_G(v)}{k}\rceil$ for all $v\in V(G)\setminus \{u\}$.
Hence the assertions follow from Corollary~\ref{cor:Devote}.
}\end{proof}
%
%
%
%
%
%
%
%
%
%
%
%
%
%
%
\subsection{Jackson-Wormald Conjecture is true}
The following theorem gives a sufficient condition
 for the existence of spanning $f$-walks.
Note that under this condition, the desired spanning $f$-walk is not necessarily closed.
For example, consider two copies of the complete graph of odd order $n$ with $n\ge 3$ and add a perfect matching $M$ between them.
The resulting connected graph $G$  does not have a spanning closed $1$-walk
 passing through the edges $M$, while satisfies $\omega(G\setminus S)\le 2$ for all vertex sets $S$.
\begin{thm}\label{thm:walk:spanning}
{Let $G$ be a connected graph and let $f$ be a positive integer-valued function on $V(G)$.
If for all $S\subseteq V(G)$,
$$\omega(G\setminus S)\le \sum_{v\in S} (f(v)-1)+2,$$
then $G$ admits a spanning $f$-walk  passing through the edges of an arbitrary given matching
(and a spanning  $f$-walk  passing through the edges of an arbitrary given  connected $(1,f+1)$-factor).
}\end{thm}
\begin{proof}
{By Corollary~\ref{cor:matching}, the graph $G$ has a spanning $(f+1)$-tree $T$ containing an arbitrary given matching
(to prove the second assertion, $T$ can play the role of  an arbitrary given  connected $(1,f+1)$-factor).
Let $G_0$ be the graph obtained from $G$ by adding a new vertex $v_0$ and joining it to all other vertices. Define $f(v_0)=2$. 
It is easy to check that for every $S\subseteq V(G_0)$, $\omega(G_0\setminus S)\le \sum_{v\in S} (f(v)-1)+1$,  regardless of  $v_0\in S$ or not.
Thus by Corollary~\ref{cor:Devote}, the graph $G_0$ contains a spanning $f$-forest $F_0$ such that $d_{F_0}(v_0)\in \{0,2\}$ and for each $v\in V(G)$, 
$d_{F_0}(v)$ and $d_T(v)$ have the same parity.
Let $F$ be the factor of $G$ obtained from $F_0$ by removing the vertex $v_0$. 
Insert a new copy of $F$ into $T$ and call the resulting connected graph $H$.
Obviously, for each 
$v\in V(G)\setminus A$, $d_H(v)$ must be even and $d_H(v)= d_F(v)+d_T(v)\le 2f(v)+1$, where $A$ is the set of neighbours of $v_0$ in $F_0$. Note that $|A|\in \{0,2\}$.
Moreover, for each 
$v\in  A$,  $d_F(v)\le f(v)-1$ and so $d_{H}(v)= d_F(v)+d_T(v)\le 2f(v)$.
Therefore, the graph $H$ admits a spanning  $f$-trail and so $G$ admits a spanning $f$-walk.
}\end{proof}
To guarantee the existence of spanning closed $f$-walks, we need to push down the upper bound in Theorem~\ref{thm:walk:spanning} only by one
 as the next theorem. 
\begin{thm}\label{thm:walk:proving-conjecture}
{Let $G$ be a graph and let $f$ be a positive integer-valued function on $V(G)$.
If for all $S\subseteq V(G)$,
$$\omega(G\setminus S)\le \sum_{v\in S} (f(v)-1)+1,$$
then $G$ admits a spanning closed $f$-walk passing through the edges of an arbitrary given matching
(and a spanning closed $f$-walk  passing through the edges of an arbitrary given connected $(1,f+1)$-factor).
}\end{thm}
\begin{proof}
{The graph $G$ must automatically be connected, because $\omega(G\setminus \emptyset )=1$. 
Thus by Corollary~\ref{cor:matching}, the graph $G$ has a spanning $(f+1)$-tree $T$ containing an arbitrary given matching
(to prove the second assertion, $T$ can play the role of  an arbitrary given  connected $(1,f+1)$-factor).
By Corollary~\ref{cor:Devote}, the graph $G$ contains a spanning $f$-forest $F$ such that for each vertex $v$, 
$d_{F}(v)$ and $d_T(v)$ have the same parity.
Insert a new copy of $F$ into $T$ and call the resulting connected graph $H$.
For each vertex $v$, $d_H(v)$ must be even and $d_H(v)= d_F(v)+d_T(v)\le 2f(v)+1$.
Therefore, the graph $H$ admits a spanning closed $f$-trail and so $G$ admits a spanning closed $f$-walk.
}\end{proof}
\begin{cor}
{A simple graph $G$ admits a spanning closed $2$-walk passing through the edges of an arbitrary given  factor $F$ with maximum degree at most $2$,
if for all $S\subseteq V(G)$, $\omega(G\setminus S)\le \frac{1}{4}|S|+1$.
}\end{cor}
\begin{proof}
{By applying Corollary~\ref{cor:path-cycle},  $F$ can be extended to connected factor $H$ with $\Delta(H)\le 3$.
Thus by Theorem~\ref{thm:walk:proving-conjecture}, the graph $G$ admits a spanning closed $2$-walk passing through the edges of $H$ and so does $F$.
}\end{proof}
%
%
%
%
%
The following corollary is an immediate consequence of Lemma~\ref{subsection:conclusions:lem} and Theorem~\ref{thm:walk:proving-conjecture}.
\begin{cor}{\rm(\cite{MR1126990})}
{Every connected $K_{1,n}$-free simple graph with $n\ge 2$ has a spanning closed $(n-1)$-walk.
}\end{cor}
The next corollary improves Theorem 4.2 in~\cite{MR1126990} and implies Corollary 3.1 in~\cite{MR2771155}.
Note that there are infinitely many $k$-connected $K_{1,n}$-free simple graphs
with $k\ge 2$ and $n\ge 3$ having no spanning closed $\lfloor \frac{n-1}{k}\rfloor$-walks,
 which were constructed by
Jin and Li~\cite{MR2037342}.
\begin{cor}
{Every $k$-connected $K_{1,n}$-free simple graph with $n\ge 2$ has a spanning closed $(\lceil \frac{n-1}{k}\rceil+1)$-walk.
}\end{cor}
\begin{proof}
{We repeat the proof of Theorem 4.2 in~\cite{MR1126990}.
Let $S$ be a cutset of $G$.
Since $G$ is $k$-connected, every component of $G\setminus S$ is joined to at least $k$ vertices in $S$.
Since $G$ is $K_{1,n}$-free, every vertex of $S$ is joined to at most $n-1$ components of $G\setminus S$.
Hence $ \omega(G\setminus S)k\le (n-1)|S|$.
Therefore, for all vertex sets $S$, $\omega(G\setminus S)\le (n-1)|S|/k+1$,  regardless of  $S$ is a cutset or not.
Hence the assertion follows from Theorem~\ref{thm:walk:proving-conjecture} with replacing $\lceil (n-1)/k\rceil +1$ instead of $f(v)$. 
}\end{proof}
The following result confirms  Conjecture 2.1 in~\cite{MR1126990}.
Note that there are infinitely many graphs with toughness approaching 
$\frac{1}{n-5/8}$ having no spanning closed $n$-walks, which were constructed 
by Ellingham and Zha~\cite{MR1740929}.
\begin{cor}\label{cor:walk:proving-conjecture}
{Every $\frac{1}{n-1}$-tough graph with $n\ge 2$ admits a spanning closed $n$-walk. 
}\end{cor}
\begin{proof}
{If $S\subseteq V(G)$, then by the assumption, we have $\omega (G\setminus S)\le \max\{1,(n-1)|S|\}\le (n-1)|S|+1$.
Thus the assertions follows from Theorem~\ref{thm:walk:proving-conjecture} with setting $f(v)=n$.
}\end{proof}
The next result confirms  Conjecture 23 in~\cite{MR1871346}.
Note that there are infinitely many $r$-edge-connected $r$-regular 
simple graphs with $r\ge 3$ having no spanning closed $1$-walks,
which were constructed by Meredith~\cite{MR0311503}.
\begin{cor}
{Every $r$-edge-connected $r$-regular graph admits a spanning closed $2$-walk.
}\end{cor}
\begin{proof}
{Apply Lemma~\ref{lem:Omega(GS)} and Corollary~\ref{cor:walk:proving-conjecture}. 
}\end{proof}
Finally, we propose the following conjecture to make a stronger version for Theorem~\ref{thm:walk:proving-conjecture}.
\begin{conj}
{Let $G$ be a graph and let $f$ be a positive integer-valued function on $V(G)$.
If for all $S\subseteq V(G)$,
$$\omega(G\setminus S)\le \sum_{v\in S} (f(v)-1)+1+\frac{1}{2}\omega(G[S]),$$
then $G$ admits a spanning closed $f$-walk passing through the edges of an arbitrary given  matching.
}\end{conj}
%
%
%
%
%
%
%
%
%
%
%
%
%
%
%
%
\section*{Acknowledgments}
The author would like to thank Ahmadreza G. Chegini and the referees for their helpful comments.

\end{document}